\documentclass[english,british]{article}
\usepackage[T1]{fontenc}
\usepackage[latin9]{inputenc}
\usepackage{verbatim}
\usepackage{float}
\usepackage{amsthm}
\usepackage{amsmath}
\usepackage{amssymb}
\usepackage{graphicx}

\makeatletter

\floatstyle{ruled}
\newfloat{algorithm}{tbp}{loa}
\providecommand{\algorithmname}{Algorithm}
\floatname{algorithm}{\protect\algorithmname}

\numberwithin{equation}{section}
\numberwithin{figure}{section}
\theoremstyle{plain}
\newtheorem{thm}{\protect\theoremname}[section]
  \theoremstyle{definition}
  \newtheorem{defn}[thm]{\protect\definitionname}
  \theoremstyle{remark}
  \newtheorem{rem}[thm]{\protect\remarkname}
  \theoremstyle{plain}
  \newtheorem{lem}[thm]{\protect\lemmaname}

\usepackage{a4wide}
\usepackage{bbm}

\newcommand{\p}{\mathbb{P}}
\newcommand{\N}{\mathbb{N}}
\newcommand{\E}{\mathbb{E}}
\newcommand{\R}{\mathbb{R}}
\newcommand{\Z}{\mathbb{Z}}
\newcommand{\1}{\mathbbm{1}}

\date{\today}

\allowdisplaybreaks

\newtheorem{theorem}{Theorem}[section]
\newtheorem{lemma}[theorem]{Lemma}

\newtheorem{hypothesis}{Hypothesis}

\newtheorem{example}[theorem]{Example}
\numberwithin{equation}{section}

\makeatother

\usepackage{babel}
  \addto\captionsbritish{\renewcommand{\definitionname}{Definition}}
  \addto\captionsbritish{\renewcommand{\lemmaname}{Lemma}}
  \addto\captionsbritish{\renewcommand{\remarkname}{Remark}}
  \addto\captionsbritish{\renewcommand{\theoremname}{Theorem}}
  \addto\captionsenglish{\renewcommand{\definitionname}{Definition}}
  \addto\captionsenglish{\renewcommand{\lemmaname}{Lemma}}
  \addto\captionsenglish{\renewcommand{\remarkname}{Remark}}
  \addto\captionsenglish{\renewcommand{\theoremname}{Theorem}}
  \providecommand{\definitionname}{Definition}
  \providecommand{\lemmaname}{Lemma}
  \providecommand{\remarkname}{Remark}
\addto\captionsenglish{\renewcommand{\algorithmname}{Algorithm}}
\providecommand{\theoremname}{Theorem}

\begin{document}

\title{Perfect simulation for the Feynman-Kac law on the path space}

\author{Christophe Andrieu%
\thanks{Department of Mathematics University of Bristol University Walk Bristol
BS8 1TW, UK, \texttt{c.andrieu@bris.ac.uk} %
}, Nicolas Chopin %
\thanks{C.R.E.S.T., Timbre J350 3, Avenue Pierre Larousse 92240 MALAKOFF,
FRANCE, \texttt{Nicolas.Chopin@ensae.fr}%
} Arnaud Doucet%
\thanks{Department of Statistics 1 South Parks Road Oxford, OX1 3TG United
Kingdom, \texttt{doucet@stats.ox.ac.uk} %
}, Sylvain Rubenthaler%
\thanks{Laboratoire J. A. Dieudonné Université de Nice - Sophia Antipolis
Parc Valrose 06108 Nice, cedex 2, FRANCE, \texttt{rubentha@unice.fr} %
}}
\maketitle
\selectlanguage{english}%
\begin{abstract}
This paper describes an algorithm of interest. This is a preliminary
version and we intend on writing a better descripition of it and getting
bounds for its complexity.
\end{abstract}
\selectlanguage{british}%

\section{Introduction}

We are given a transition kernel $M$ (on a space $E$), $M_{1}$
a probability measure on $E$ and potentials $(G_{k})_{k\geq1}$ ($G_{k}:E\rightarrow\R_{+}$).
All densities and kernels are supposed to have a density with respect
to some reference measure $\mu$ on $E$. We want to draw samples
according to the law (on paths of length $P$)
\[
\pi(f)=\frac{\mathbb{E}(f(X_{1},\dots,X_{P})\prod_{i=1}^{P-1}G_{i}(X_{i})}{\mathbb{E}(\prod_{i=1}^{P-1}G_{i}(X_{i})}
\]
where $(X_{k})$ is Markov with initial law $M_{1}$ and transition
$M$. For all $n\in\N^{*}$, we note $[n]=\{1,\dots,n\}$.

\section{Densities of branching processes}

\subsection{Description of a branching system}

We start with $n_{1}$ particles (i.i.d. with law $M_{1}$, $n_{1}$
is a fixed number). We then proceed recursively through time. If we
have $N_{i}$ particules at time $k$, the system evolves in the following
manner:
\begin{itemize}
\item The number of childern of $X_{k}^{i}$ (the $i$-th particle at time
$k$) is a random variable $A_{k+1}^{i}$ with law $f_{k+1}$ such
that : $\p(A_{k+1}^{i}=j)=f_{k+1}(G_{k}(X_{k}^{i}),j)$ (here, $f_{k}$
is a law with a parameter $G_{k}(X_{k}^{i})$, we will define this
law later). The variables $A_{k+1}^{i}$ ($1\leq i\leq N_{k}$) are
independent. We then have $N_{k+1}=\sum_{i=1}^{N_{k}}A_{k+1}^{i}$
\item If $N_{k+1}\neq0$, we draw $\sigma_{k+1}$ uniformly in $\mathcal{S}_{N_{k+1}}$
(the $N_{k+1}$-th symmetric group). If $N_{k+1}=0$, we use the convention
$\mathcal{S}_{N_{k+1}}=\emptyset$, we do not draw $\sigma_{k+1}$
and the system stops there. 
\item We set $\forall i\in[N_{k}]$, $C_{k+1}^{i}=\{A_{k+1}^{1}+\dots+A_{k+1}^{i-1},\dots,A_{k+1}^{1}+\dots+A_{k+1}^{i-1}+A_{k+1}^{i}\}$.
If $j\in\sigma_{k+1}(C_{k+1}^{i})$, we draw $X_{k+1}^{j}\sim M(X_{k}^{i},.)$.
We say that $i$ at time $k$ is the father of $j$ at time $k+1$.
We will denote this relation by the symbol $(i,k)\rightsquigarrow(j,k+1)$.\end{itemize}
\begin{defn}
\label{def:ancestor}If we have integers $k<k'$ and a sequence $(i_{k},i_{k+1},\dots,i_{k'})$
such that $(i_{k},k)\rightsquigarrow(i_{k+1},k+1)\rightsquigarrow\dots\rightsquigarrow(i_{k'},k')$,
we will say that the particle $i_{k}$ at time $k$ is ancestor of
$i_{k'}$ at time $k'$ and we will write $(i_{k},k)\rightsquigarrow(i_{k'},k')$.
We define $P'=\sup\{1\leq k\leq P,N_{k}\neq0\}$ (it is a random variable).
\end{defn}
Such a system has a density on the space 
\begin{multline*}
\{(n_{2},\dots,n_{p},x_{k}^{i},a_{k}^{i},s_{k}):n_{2},\dots,n_{P}\in\N,\\
x_{k}^{i}\in E(1\leq k\leq P,1\leq i\leq n_{k}),a_{k}^{i}\in\N(2\leq k\leq P,1\leq i\leq n_{k}),s_{k}\in\mathcal{S}_{N_{k}}(2\leq k\leq P)\}\,.
\end{multline*}
 This density is equal to : 
\begin{multline*}
q_{0}(n_{2},\dots,n_{P},(a_{k}^{i})_{2\leq k\leq p',1\leq i\leq N_{k}},(x_{k}^{i})_{1\leq k\leq p',1\leq i\leq n_{k}},(s_{n})_{2\leq n\leq p'})\\
=\prod_{i=1}^{N_{1}}M_{1}(x_{1}^{i})\prod_{k=2}^{P}\prod_{i=1}^{N_{k-1}}f_{k}(G_{k-1}(x_{k-1}^{i}),a_{k}^{i})\frac{1}{n_{k}!}\prod_{j\in s_{k}(B_{k}^{i})}M(x_{k-1}^{i},x_{k}^{j})\,,
\end{multline*}
with the convention $\prod_{i\in\emptyset}\dots=1$. In all this article,
we use the letters $X$, $A$, $B$, $N$, $P'$, $\sigma$ for random
variables (respectively in $E$, $\N$, $\N$, $\N$, $\N$, some
permutation set) and we will use the letters $x$, $a$, $b$, $n$,
$p'$, $s$ when describing densities for these variables. 
\begin{rem}
The random permutations $\sigma_{k}$ ease the writing of the formulas
but have no deep signification. Their technical purpose is to ensure
the densities $q$, $\widehat{\pi}$ defined below are mutually absolutly
continuous on some set of non-zero measure (for $q$ and $\widehat{\pi}$).
\end{rem}

\subsection{Proposal density}

We take the above branching system and we draw a path by drawing a
number $i$ uniformly in $\{1,\dots,N_{P}\}$. We take a random number
$B_{P'}$ uniformly in $[N_{P'}]$. We then look at the ancestral
path of $X_{N_{P'}}^{B_{P'}}$, meaning we build recursively backwards
$(B_{k})_{1\leq k\leq P'}$ by taking for all $k$ in $\{1,\dots,P'\}$,
$B_{k}$ such that $B_{k+1}\in\sigma_{k+1}(C_{k+1}^{B_{k}})$. We
obtain a branching system containing one special trajectory $(X_{1}^{B_{1}},\dots,X_{P'}^{B_{P'}})$.
$ $This random variable lives in the following space (with $p'=\sup\{1\leq k\leq P:N_{k}\neq0\}$).
\begin{multline}
\{(n_{2},\dots,n_{P},x_{k}^{i},a_{k}^{i},s_{k},b_{i}):n_{2},\dots,n_{P}\in\N,\\
x_{k}^{i}\in E(1\leq k\leq p',1\leq i\leq n_{k}),a_{k}^{i}\in\N(2\leq k\leq p',1\leq i\leq n_{i}),\\
s_{k}\in\mathcal{S}_{n_{k}}(2\leq k\leq p'),b_{i}\in[n_{i}](1\leq i\leq p')\}\,,\label{eq:espace}
\end{multline}
and have a density $q$ satisfying:
\begin{multline*}
q(n_{2},\dots,n_{P},(a_{k}^{i})_{2\leq k\leq P,1\leq i\leq n_{k}},(x_{k}^{i})_{1\leq k\leq P,1\leq i\leq n_{k}},(s_{k})_{2\leq k\leq P},(b_{k})_{1\leq k\leq P})\times\1_{p'=P}\\
=\frac{1}{n_{P}}q_{0}(n_{2},\dots,n_{P},(a_{k}^{i})_{2\leq k\leq P,1\leq i\leq n_{k}},(x_{k}^{i})_{1\leq k\leq P,1\leq i\leq n_{k}},(s_{k})_{2\leq k\leq P})\times\1_{p'=P}\,.
\end{multline*}
What happens precisely outside the set $\{p'=P\}$ is not useful to
us. We define ancestry relationships in this system as in Definition
\ref{def:ancestor}.

\subsection{Target law}

We draw a trajectory $(Y_{1},\dots,Y_{P})$ with the law $\pi$ then
a branching system conditioned on containing the trajectory $(Y_{1},\dots,Y_{P})$.
The order of operations is as followed
\begin{itemize}
\item Draw $(Y_{1},\dots,Y_{P})$ with law $\pi(.)$.
\item We draw $B_{1}$ uniformly in $[N_{1}]$, we set $X_{1}^{B_{1}}=Y_{1}$.
We draw $(X_{1}^{i})_{1\leq i\leq n_{1},i\neq b_{1}}$ i.i.d. variables
of law $M_{1}$.
\item If we have the $(k-1)$-th generation, we draw $A_{k}^{B_{k-1}}$
with law $f_{k}(G_{k-1}(X_{k-1}^{B_{k-1}}),.)$ conditioned to be
in $\N^{*}$ (we call this law $\widehat{f}(G_{k-1}(X_{k-1}^{B_{k-1}}),.)$).
For $i\in N_{k-1}$, $i\neq B_{k-1}$, we draw $A_{k}^{i}\sim f_{k}(G_{n-1}(x_{k-1}^{B_{k-1}}),.)$.
We set $N_{k}=\sum_{i=1}^{N_{k-1}}A_{k}^{i}$. We draw $\sigma_{k}$
uniformly in $\mathcal{S}_{N_{k}}$. We set $B_{k}=\sigma_{k}(A_{k}^{1}+\dots+A_{k}^{b_{k-1}-1}+1)$,
$X_{k}^{B_{k}}=Y_{k}$. For $j\in[N_{k}]$, if $j\neq B_{k}$ and
$j\in\sigma_{k}(C_{k}^{i})$ ($C_{k}^{i}=\{A_{k}^{1}+\dots+A_{k}^{i-1}+1,\dots,A_{k}^{1}+\dots+A_{k}^{i}\}$),
we draw $X_{k}^{j}\sim M(X_{k-1}^{i},.)$.
\end{itemize}
We get a variable in the following space 
\begin{multline*}
\{(n_{2},\dots,n_{P},x_{k}^{i},a_{k}^{i},\sigma_{k},b_{i}):n_{2},\dots,n_{P}\in\N^{*},x_{k}^{i}\in E(1\leq k\leq P,1\leq i\leq n_{k}),\\
A_{k}^{i}\in\N(1\leq k\leq P,1\leq i\leq n_{k}),s_{k}\in\mathcal{S}_{N_{n}}(2\leq n\leq P),b_{i}\in[N_{i}](1\leq i\leq n)\}\,,
\end{multline*}
with the following density:
\begin{multline}
\widehat{\pi}(n_{2},\dots,n_{P},(a_{k}^{i})_{2\leq k\leq P,1\leq i\leq N_{k}},(x_{k}^{i})_{1\leq k\leq P,1\leq i\leq N_{k}},(s_{k})_{2\leq k\leq P},(b_{k})_{1\leq k\leq P}))\\
=\pi(x_{1}^{b_{1}},\dots,x_{P}^{b_{P}})\frac{1}{N_{1}}\prod_{1\leq i\leq N_{1},i\neq b_{1}}M_{1}(x_{1}^{i})\\
\prod_{k=2}^{P}\left(\widehat{f}_{k}(G_{k-1}(x_{k-1}^{b_{k-1}}),a_{k}^{b_{k-1}})\prod_{1\leq i\leq N_{k-1},i\neq b_{k-1}}f_{k}(G_{k-1}(x_{k-1}^{i}),a_{k}^{i})\right.\\
\left.\times\frac{1}{n_{k}!}\prod_{1\leq i\leq N_{k-1}}\prod_{j\in s_{k}(B_{k}^{i}),,j\neq b_{k}}M(x_{k-1}^{i},x_{k}^{j})\right)\,.\label{eq:densite-cible}
\end{multline}
Notice that: ($\forall z,j$) $\widehat{f}_{k}(g,j)=\frac{f_{k}(g,j)}{1-f_{k}(g,0)}$
( $x_{k-1}^{b_{k-1}}$ is conditioned on having at least one children).
We define ancestry relationships in this system as in Definition \ref{def:ancestor}.

\subsection{Ratio of the densities}

We write the ratio $\widehat{\pi}/q$ and we get:

\begin{multline*}
\frac{\widehat{\pi}(N_{2},\dots,N_{P},(A_{k}^{i})_{1\leq k\leq P-1,1\leq i\leq N_{k}},(x_{k}^{i})_{1\leq k\leq P,1\leq i\leq N_{k}},(s_{k})_{2\leq k\leq P},(b_{k})_{1\leq k\leq P}))}{q(N_{2},\dots,N_{P},(A_{k}^{i})_{1\leq k\leq P-1,1\leq i\leq N_{k}},(x_{k}^{i})_{1\leq k\leq P,1\leq i\leq N_{k}},(s_{k})_{2\leq k\leq P},(b_{k})_{1\leq k\leq P}))}\\
=\pi(x_{1}^{b_{1}},\dots,x_{P}^{b_{P}})\times\frac{N_{P}}{N_{1}}\times\frac{1}{M_{1}(x_{1}^{b_{1}})\prod_{k=2}^{P}M(x_{k-1}^{b_{k-1}},x_{k}^{b_{k}})}\times\prod_{k=2}^{P}\frac{\widehat{f}_{k}(G_{k-1}(x_{k-1}^{b_{k-1}}),A_{k}^{b_{k-1}})}{f_{k}(G_{k-1}(x_{k-1}^{b_{k-1}}),A_{k}^{b_{k-1}})}\,.
\end{multline*}
Let us take $f_{k}$ such that for all $g,i$ ($i\neq0$), $\frac{\widehat{f}_{k}(g,i)}{f_{k}(g,i)}=\frac{\Vert G_{k-1}\Vert_{\infty}}{g}$.
This means that $1-f_{k}(g,0)=\frac{g}{\Vert G_{k-1}\Vert_{\infty}}$.
We then get:

\begin{equation}
\frac{\widehat{\pi}(\dots)}{q(\dots)}=\frac{N_{P}\prod_{i=2}^{P}\Vert G_{i-1}\Vert_{\infty}}{N_{1}Z}\,,\label{eq:ratio-1}
\end{equation}
with $Z=\E(\prod_{k=1}^{P-1}G_{k}(X_{k}))$ ($(X_{k})_{k\geq1}$ is
a Markov chain with initial law $M_{1}$ and kernel transition $M$).

\section{Perfect simulation algorithm}

\subsection{Stability of the branching process}

We want the branching process to be stable. So we need that 
\begin{equation}
\frac{1}{N_{k-1}}\sum_{i=1}^{N_{k-1}}\sum_{j=1}^{+\infty}jf_{k}(G_{k-1}(X_{k-1}^{i}),j)\mbox{ be of order 1 (\ensuremath{\forall}k).}\label{eq:cond-01}
\end{equation}
 Let us take (for some $q_{k}$), 
\[
f_{k}(g,0)=1-\frac{\Vert G_{k}\Vert_{\infty}}{\beta_{k}}\,,\, f_{k}(g,i)=\frac{\Vert G_{k}\Vert_{\infty}}{q_{k}\beta_{k}}\mbox{ for }1\leq i\leq q_{k}\,.
\]
 We then get $\sum_{i=1}^{q_{k}}i\times f_{k}(g,i)=\frac{(q_{k}+1)g}{2\beta_{k}}$.
So it is sensible to fix $q_{k}$ such that 
\begin{equation}
\Vert G_{k}\Vert_{\infty}=\frac{q_{k}+1}{2}\times\frac{1}{N}\sum_{i=1}^{N}G_{k-1}(\overline{X}_{k-1}^{i})\label{eq:stabilite}
\end{equation}
where $(\overline{X}_{k-1}^{i})$ is a sequential Monte-Carlo system
with $N$ particles, this has to be computed beforehand. Simulations
show that this procedure indeed gives you stable branching processes.

\subsection{\label{sub:Markovian-transition}Markovian transition}

Let us now describe a Markovian transition in $E^{P}$. The integer
$N_{1}$ is fixed. We start in $(z_{1},\dots,z_{P})\in E^{P}$. 
\begin{enumerate}
\item We draw $N_{2},\dots,N_{P}$, $(X_{k}^{i})_{1\leq k\leq P,1\leq i\leq N_{k},i\neq B_{k}}$,
$(A_{k}^{i})_{1\leq k\leq P,1\leq i\leq n_{k}}$, $(S_{k}\in\mathcal{S}_{N_{k}})_{2\leq k\leq P}$,
$(B_{k})_{1\leq k\leq P}$ with the density 
\begin{equation}
\frac{\widehat{\pi}(\dots,z_{1},\dots,z_{P},\dots)}{\pi(z_{1},\dots,z_{P})}\label{eq:cond-SMC}
\end{equation}
 ($z_{1},\dots z_{P}$ in place of $x_{1}^{b_{1}},\dots,x_{P}^{b_{P}}$
in equation (\ref{eq:densite-cible})). This amounts to drawing a
genealogy conditionned to contain $(z_{1},\dots,z_{P})$. Let us set
$\forall k\in\{1,\dots,P\}$, $X_{k}^{B_{k}}=z_{k}$. Let $\mathcal{X}$
be the variable containing all the $N_{k},\, X_{k}^{i},\, A_{k}^{i},\, S_{k},\, B_{k}$.
We say that $\chi$ is a conditionnal forest. 
\item We draw $\overline{N}{}_{2},\dots,\overline{N}{}_{P}$ $(\overline{X}{}_{k}^{i})_{1\leq k\leq P,1\leq i\leq\overline{N}_{k}}$,
$(\overline{A}{}_{k}^{i})_{1\leq k\leq P,1\leq i\leq\overline{N}_{k}}$,
$(\overline{S}{}_{k}\in\mathcal{S}_{N_{k}})_{2\leq k\leq P}$, $(\overline{B}{}_{k})_{1\leq k\leq P}$
with density $q(.)$. We denote by $\overline{\mathcal{X}}$ the corresponding
variable. We say that $\overline{\chi}$ is a proposal. 
\item We draw $V$ uniform on $[0,1]$. If $ $ $V\leq\alpha(\chi,\overline{\chi}):=\inf\left(1,\frac{\widehat{\pi}(\overline{\mathcal{X}})q(\mathcal{X})}{\widehat{\pi}(\mathcal{X})q(\overline{\mathcal{X}})}\right)$,
we set $(\overline{Z}{}_{1},\dots,\overline{Z}{}_{P})=(\overline{X}{}_{1}^{B_{1}},\dots,\overline{X}_{P}{}^{B_{P}})$,
and if not, we set $(\overline{Z}{}_{1},\dots,\overline{Z}{}_{P})=(z_{1},\dots,z_{P})$. \end{enumerate}
\begin{rem}
In the case $\overline{N}_{P}=0$, we then have $q(\overline{\chi})=0$,
$\widehat{\pi}(\chi)\neq0$ and so $\alpha(\chi,\overline{\chi})=0$. \end{rem}
\begin{lem}
The transformation of $(z_{1},\dots,z_{P})$ into $(\overline{Z}_{1},\dots,\overline{Z}_{P})$
is a Metropolis Markov kernel (on $E^{P}$) for which $\pi$ is invariant
(much in the spirit of \cite{andrieu-doucet-holenstein-2010}). \end{lem}
\begin{proof}
Suppose that we draw $(Z_{1},\dots,Z_{P})\sim\pi$ and we use it as
a starting point in the transition described above. Then $\chi\sim\widehat{\pi}$.
Wet set $\overline{\overline{\chi}}=\overline{\chi}\1_{V\leq\alpha(\chi,\overline{\chi})}+\chi\1_{V>\alpha(\chi,\overline{\chi})}$.
By Metropolis, we have that $\overline{\overline{\chi}}\sim\widehat{\pi}$.
And so $ $$(\overline{Z}{}_{1},\dots,\overline{Z}{}_{P})\sim\widehat{\pi}$.
\end{proof}
Due to (\ref{eq:ratio-1}), we have 
\begin{equation}
\frac{\widehat{\pi}(\overline{\mathcal{X}})q(\mathcal{X})}{\widehat{\pi}(\mathcal{X})q(\overline{\mathcal{X}})}=\frac{\overline{N}_{P}}{N_{P}}\,.\label{eq:rapport-acceptation}
\end{equation}

\subsection{Backward coupling}

We know want to use a backward coupling algorithm (as in \cite{foss-tweedie-1998,propp-wilson-1996}).
Any $U_{-i}$ is sufficient to make a simulation of the Markovian
transition above. We can always parametrize the Markov transition
above on the following way 
\begin{equation}
(\overline{Z}{}_{1},\dots,\overline{Z}{}_{P})=F_{U}(z_{1},\dots,z_{P})\,,\, U\sim\mu\,,\label{eq:parametrization}
\end{equation}
with some law $\mu$ on a big space $E'$ and $F$ above is a function
of $U$ and $(z_{1},\dots,z_{P})$ ($U$ is written as an index because
it will be more convenient later). We will write $\chi(z_{1},\dots,z_{P},U)$,
$\overline{\chi}(U)$, $V(U)$ for, respectively, the conditionnal
forest, the proposal and the uniform variable appearing in the building
of $ $$(\overline{Z}{}_{1},\dots,\overline{Z}{}_{P})$. We remark
here that $\overline{\chi}(U)$ and $V(U)$ do not depend on $(z_{1},\dots,z_{P})$.
We will write $N_{P}(z_{1},\dots,z_{P},U)$ for the size of the population
in $\chi(z_{1},\dots,z_{P},U)$ at time $P$ and $\overline{N}_{P}$
for the size of the population in $\overline{\chi}(U)$ at time $P$.

By Theorem 3.1 of \cite{foss-tweedie-1998}, if $T$ is a stopping
time, relatively to the filtration $(\sigma(U_{0},\dots,U_{-i}))_{i\geq0}$,
such that $\forall(z_{1}^{(1)},\dots,z_{P}^{(1)}),(z_{1}^{(2)},\dots,z_{P}^{(2)})\in E^{P},$
$F_{U_{0}}\circ\dots\circ F_{U_{-T}}(z_{1}^{(1)},\dots,z_{P}^{(1)})=F_{U_{0}}\circ\dots\circ F_{U_{-T}}(z_{1}^{(2)},\dots,z_{P}^{(2)})$,
then $F_{U_{0}}\circ\dots\circ F_{-T}(z_{1}^{(1)},\dots,z_{P}^{(1)})$
is exactly of law $\pi$. So we have the following lemma. 
\begin{lem}
Suppose we have a bound of the form $ $$ $$ $$N_{P}(z_{1},\dots,z_{P},u)\leq M(u)$,
for all $(z_{1},\dots,z_{P})$ in $\mathcal{E}_{P}:=\{(X_{1}(\omega),\dots,X_{P}(\omega)),\,\omega\in\Omega\}$,
$\forall u\in E'$. Let 
\[
T=\inf\left\{ i:V(U_{-i})\leq\frac{\overline{N}_{P}(U_{-i})}{M(U_{-i})}\right\} \text{.}
\]
Then, for all $(z_{1},\dots,z_{P})\in\mathcal{E}_{P}$, $F_{U_{0}}\circ\dots\circ F_{U_{-T}}(z_{1},\dots,z_{P})\sim\pi$.
\end{lem}
Under the assumption of the above lemma, we can implement the following
algorithm.

\begin{algorithm}[H]
i=-1

Repeat

~~~~i=i+1

~~~~draw $U_{-i}$

~~~~compute $R_{-i}=\overline{N}_{P}(U_{-i})/M(U_{-i})$

Until ($V(U_{-i})\leq R_{-i}$)

Take any $(z_{1},\dots,z_{P})\in\mathcal{E}_{P}$

Display $F_{U_{0}}\circ\dots\circ F_{U_{-i}}(z_{1},\dots,z_{P})$

\caption{Perfect simulation algorithm}
\end{algorithm}

\section{Applications}

\subsection{Random directed polymers}

This kind of model is described in \cite{den-hollander-2009} (chapter
12), \cite{bezerra-tindel-viens-2008}. Let $(X_{k})_{k\geq1}$ be
a symmetric simple random walk in $\Z$ with $X_{1}=0$. We are given
i.i.d. variables $(\xi_{k,i})_{k\geq1,i\in\Z}$ with Bernoulli law
of parameter $p>0$. These variables are called ''random environment''.
We set (random) potentials : $V_{k}(i)=\exp(-\beta\xi_{k,i})$ ($\beta>0$)
and we are interested in the following quenched law (quenched means
the $\xi_{k,i}$ are fixed, the $\E_{X}$ below means we are taking
expectation on the $X_{i}$'s only and not on the $\xi_{n,i}$'s)
: 
\[
\pi_{1:P}(f)=\frac{\E_{X}(f(X_{1:P})\prod_{i=1}^{P-1}V_{i}(X_{i}))}{\E_{X}(\prod_{i=1}^{P-1}V_{i}(X_{i}))}\,.
\]

For a trajectory $(z_{1},\dots,z_{P})\in\mathcal{E}_{P}$ and a random
variable $U$ we draw a conditionnal forest $\chi(z_{1},\dots,z_{P},U)$
as in \ref{sub:Markovian-transition}. We denote ($\forall i\in[P-1]$)
by $N_{P}^{(i)}(z_{i},U)$ the number of particles at time $P$ different
from the particle $B_{P}$ and having $ $$(B_{i},i)$ as an ancestor.
This is indeed a function independant of $z_{1},\dots,z_{i-1},z_{i+1},\dots,z_{n}$.
The trajectory of a simple random walk in $\Z$ is such that 
\[
M_{P}^{(i)}(U):=\sup\{N_{P}^{(i)}(z_{i},U),(z_{1},\dots,z_{P})\in\mathcal{E}_{P}\}=\sup\{N_{P}^{(i)}(z_{i},U),z_{i}\in\{-(i-1),-(i-1)+2,\dots,i-1\}<\infty\,.
\]
We denote by $N_{P}^{(c)}(U)$ the number of particles at time $P$
different from the particle $B_{P}$ and not having any $(B_{i},i)$
as an ancestor. This is indeed a function independant of $(z_{1},\dots,z_{P})$.
We have $N_{P}=1+\sum_{i=1}^{P-1}N_{P}^{(i)}(z_{i},U)+N_{P}^{(c)}(U)$.
And so we can bound
\[
N_{P}\leq1+\sum_{i=1}^{P-1}M_{P}^{(i)}(U)+N_{P}^{(c)}(U)\,.
\]
Figures \ref{fig:surdiffusion-01} and \ref{fig:surdiffusion-02}
show two numerical experiments, in two different random environments.
They illustrate in particular the weak disorder behavior explained
in \cite{den-hollander-2009}, chapter 12. 
\begin{figure}[h]
\centering{}\label{fig:surdiffusion-01}\includegraphics[scale=1]{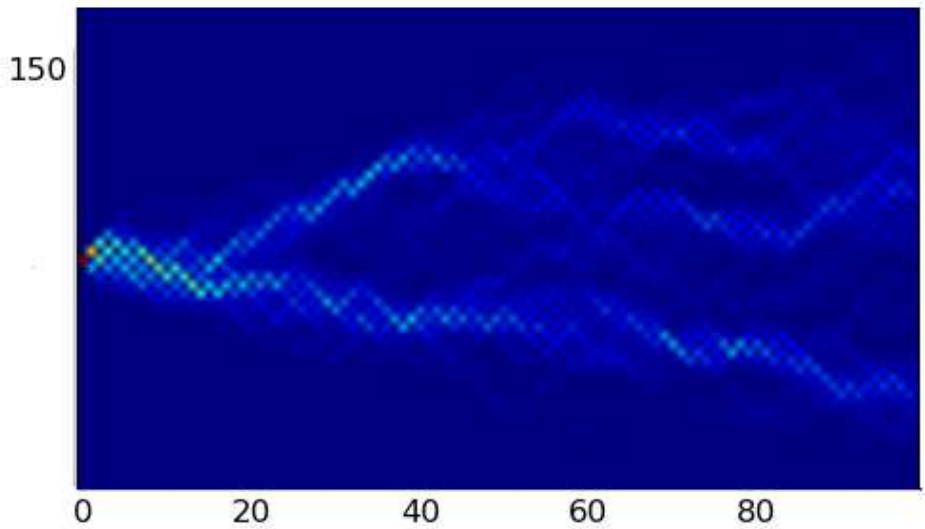}\caption{Random directed polymer, 100 trajectories. }
\end{figure}

\begin{figure}[h]

\centering{}\label{fig:surdiffusion-02}\includegraphics[scale=1]{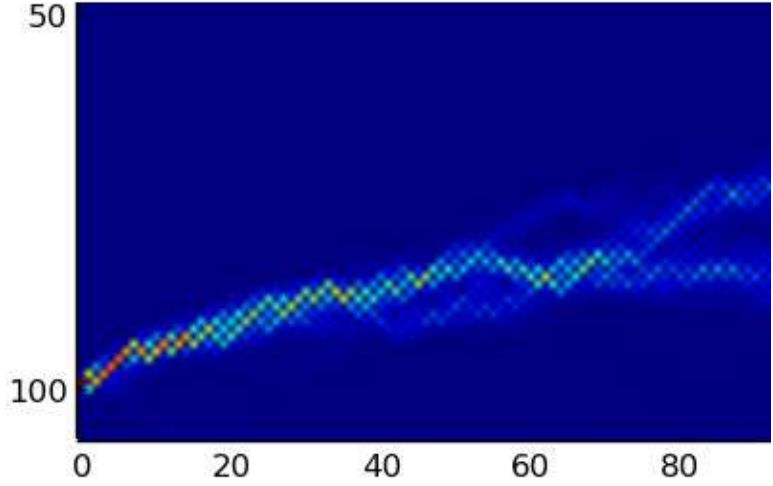}\caption{Random directed polymer, 500 trajectories}
\end{figure}

If we take the expectation over all the variable: $\E(\mbox{max de la traj. sous }\pi_{1:n})$
behaves as $n^{\zeta}$ with $\zeta\neq1/2$.Using our algorithm,
we can simulate trajectories under the law $\pi_{1:P}$ (for fixed
$\xi$, $P\in\N^{*}$). The research of the ancestors having the biggest
number of descendants at time $P$ makes that the computational cost
is $P^{2}.$ Here is the drawing of $\E(\max\dots)$ as a function
of $n$ in a $\log$-$\log$ scale (the blue line has gradient $2/3$,
the green line has gradient

\subsection{Continuous state space}

\begin{hypothesis}\label{hyp:1}

There exists a function $f:\R\rightarrow\R$ such that: for all $x_{1}\in E$,
$i\in\N$, $U_{-i}$ fixed, $(x_{1},\dots,x_{P})$ trajectory drawn
with transitions $M$ using the variables $U_{-i}$ (which we will
denote by $x_{j+1}=M_{U_{-i}}(x_{j})$, $\forall j\in\{2,\dots,P\}$),
for all $S\epsilon>0$, $\forall j\in\{2,\dots,P\}$, $\mbox{diam}(M_{U_{-i}}^{\circ(j-1)}(B_{\epsilon}(x_{1})))\leq f^{\circ(j-1)}(\epsilon)$.

\end{hypothesis}

\begin{example}If the transition $M$ is (for some constants $a,b$)
:
\[
M(x,dy)=\frac{1}{\sqrt{2\pi b^{2}}}\exp\left(-\frac{(y-ax)^{2}}{2b^{2}}\right)\,,
\]
then we can take $f(x)=ax$.\end{example}

We now want to bound the number of descendants generated by the trajectory
$(z_{1},\dots,z_{P})$ during the conditional drawing using the variables
$U_{-i}$. %
{} Let us precise how we do this conditional drawing $(z_{1},\dots,z_{P})$.
We fix $\forall n$, $\beta_{k}=\Vert G_{k}\Vert_{\infty}$ and $q_{k}$
satisfying (\ref{eq:stabilite}). For $g\in[0;\Vert G\Vert_{\infty}]$,
we set $u\mapsto F_{k,g}^{-1}(u)$ to be pseudo-inverse of the cumulative
distribution function of the law $f_{k}(g,.)$ and we set $u\mapsto\widehat{F}_{k,g}^{-1}(u)$
to be the pseudo-inverse of the cumulative distribution function of
the law $\widehat{f}_{k}(g,.)$. We are given a family $(V_{\mathbf{u}},W_{\mathbf{u}})_{\mathbf{u}\in(\N^{*})^{[k]},k\geq1}$
(random variables indexed by infinite sequences of $\N^{*}$) of independent
variables of law $\mathcal{U}([0;1])$. We are given %
{} $(\sigma_{k,N})_{k,N\geq1}$ independent variables such that $\forall k,N$,
$\sigma_{k,N}$ is uniform in $\mathcal{S}_{N}$. Suppose there exists
$M':E\times[0;1]\rightarrow E$ such that if $U\sim\mathcal{U}([0;1])$,
$x\in E$ then $M'(x,U)\sim M(x,dy)$. Suppose there exists $M_{1}':[0;1]\rightarrow\R$
tsuch that if $U\sim\mathcal{U}([0;1])$, then $M_{1}'(U)\sim M_{1}(dx)$.
The simulation goes as follows.
\begin{itemize}
\item We set $X_{1}^{i}=M_{1}'(V_{(i)})$ ($(i)$ is a sequence of length
$1$ taking value $i$) for all $i\in[N_{1}]\backslash\{1\}$, and
$X_{1}^{1}=z_{1}$. We define $\Psi_{1}:[N_{1}]\rightarrow\left(\N^{*}\right)^{[1]}$
by $\Psi_{1}(i)=(i)$. 
\item Suppose we have made the simulation up to time $k<P$ and we have
a function $\Psi_{k}:[N_{k}]\rightarrow(\N^{*})^{[N_{k}]}$ (describing
the genealogy of the particles, $\Psi_{k}(i)$ is the complete ancestral
line of particle $i$). \begin{itemize} \item For $i\in[N_{k}]\backslash\{1\}$,
we set $A_{k+1}^{i}=F_{k,G_{k}(X_{i}^{k})}^{-1}(W_{\Psi_{k}(i)})$;
\item and if $i=1$, then $X_{k}^{1}=z_{k}$, and we set $A_{k+1}^{i}=\widehat{F}_{k+1,G_{k}(z_{k})}^{-1}(W_{\Psi_{k}(i)})$.
\end{itemize} We set $N_{k+1}=\sum_{i=1}^{N_{k}}A_{k+1}^{i}$.\begin{itemize} \item 
For $j\in[N_{k+1}]\backslash\{1\},$ if $A_{k+1}^{1}+\dots+A_{k+1}^{i-1}<j\leq A_{k+1}^{1}+\dots+A_{k+1}^{i}$,
we set $\Psi_{k+1}(j)=(\Psi_{k}(i),j-(A_{k+1}^{1}+\dots+A_{k+1}^{i-1}))$,
$X_{k+1}^{j}=M'(X_{k}^{i},V_{\Psi_{k}(j)})$, \item and if $j=1$,
we set $X_{k+1}^{j}=z_{k+1}$, $\Psi_{k+1}(j)=(1,1,\dots,1)$.\end{itemize}
\item We then set $\overline{X}_{1}^{i}=X_{1}^{\sigma_{1,N_{1}}(i)}$ ($1\leq i\leq N_{1})$,
$B_{1}=\sigma_{1,N_{1}}^{-1}(1)$. We then proceed by recurrence.
If we have $(\overline{X}_{j}^{i})_{1\leq j\leq k,1\leq i\leq N_{k}}$,
$(\overline{A}_{j}^{i})_{2\leq j\leq k,1\leq i\leq N_{j-1}}$, $(\sigma_{j})_{2\leq j\leq k}$,
$B_{1},\dots,B_{n}$ with $\overline{X}_{j}^{i}=X_{j}^{\sigma_{j,N_{j}}(i)}$
($\forall j\in[k]$, $i\in[N_{j}]$) then:\\
We set $\overline{A}_{k+1}^{i}=A_{k+1}^{\sigma_{k,N_{k}}(i)}$, $B_{k+1}^{i}=\{A_{k+1}^{1}+\dots+A_{k+1}^{i-1}+1,\dots,A_{k+1}^{1}+\dots+A_{k+1}^{i}\}$,
$\sigma_{k+1}=\sigma_{k+1,N_{k+1}}^{-1}$, $\overline{B}_{k+1}^{i}=\sigma_{k+1}(B_{k+1}^{\sigma_{k,N_{k}}(i)})$,
$\overline{X}_{k+1}^{i}=X_{k+1}^{\sigma_{k+1,N_{k+1}}(i)}$, ($\forall i\dots$).
We have 

\begin{itemize}
\item if $i\in\overline{B}_{k}^{q}=\sigma_{k+1,N_{k+1}}^{-1}(B_{k+1}^{\sigma_{k,N_{k}}(q)})$
and $i\neq B_{k+1}:=\sigma_{k+1,N_{k+1}}^{-1}(1)$, $\sigma_{k+1,N_{k+1}}(i)\in B_{k+1}^{\sigma_{k,N_{k}}(q)}$,
$X_{k+1}^{\sigma_{k+1,N_{k+1}}(i)}=M'(X_{k}^{\sigma_{k,N_{k}}(q)},V_{\Psi_{k}(\sigma_{k,N_{k}}(q))})$,
then $\overline{X}_{k+1}^{i}=M'(\overline{X}_{k}^{q},V_{\Psi_{n}(\sigma_{k,N_{k}}(q))})$ 
\item and in the case $i=B_{k+1}$, $\overline{X}_{k+1}^{B_{k+1}}=X_{k+1}^{1}=z_{k+1}$.
\end{itemize}

And we have
\begin{itemize}
\item if $B_{k+1}\notin\overline{B}_{k+1}^{i}$, then $\#\overline{B}_{k+1}^{i}=\#B_{k+1}^{\sigma_{k,N_{k}}(i)}=A_{k+1}^{\sigma_{k,N_{k}}(i)}=F_{k,G_{k}(\overline{X}_{k}^{i})}^{-1}(W_{\Psi_{k}(\sigma_{k,N_{k}(i)})})$,
\item if $B_{k+1}\in\overline{B}_{k+1}^{i}$,then $\#\overline{B}_{k+1}^{i}=\widehat{F}_{k,G_{k}(\overline{X}_{k}^{i})}^{-1}(W_{\Psi_{k}(\sigma_{k,N_{k}(i)})})$.
\end{itemize}
\end{itemize}
This procedure draw $(\overline{X}_{k}^{i},\overline{A}_{k}^{i},B_{k},\sigma_{k})$
with the density (\ref{eq:cond-SMC}) (in practice, one can get rid
of the simulation of the permutations since they have no influence
on the trajectories we are interested in). We will note $(X_{k}^{i},A_{k}^{i},B_{k},\sigma_{k},k\in\dots)=\Phi((z_{i})_{i\in[P]},\left(V_{\mathbf{u}},W_{\mathbf{u}}\right)_{\mathbf{u}\in(\N^{*})^{[k]},k\geq1},(G_{k})_{1\leq k\leq P})$.

\begin{lemma}\label{lem:maj-pot}If in the procedure above, we replace
$A_{k+1}^{i}=\widehat{F}_{k+1,G_{k}(z_{k})}^{-1}(W_{\Psi_{k}(i)})$
(in the case $\Psi_{k}(i)=(N_{1},1,\dots,1)$) by $\widetilde{A}_{k+1}^{i}=\widehat{F}_{k+1,H_{k}(z_{k})}^{-1}(W_{\Psi_{k}(i)})$
for some function $H_{k}\geq G_{k}$, then we get a different system,
which we will note with $\widetilde{\,}$, 
\[
(\widetilde{X}_{k}^{i},\widetilde{A}_{k}^{i},\widetilde{B}_{k},\widetilde{\sigma}_{k},k\in\dots)=\Phi((z_{i})_{i\in[P]},\left(V_{\mathbf{u}},W_{\mathbf{u}}\right)_{\mathbf{u}\in(\N^{*})^{[k]},k\geq1},(H_{k})_{1\leq k\leq P}))\,,
\]
 such that $\forall k$, $\{X_{k}^{i},1\leq i\leq N_{k}\}\subset\{\widetilde{X}_{k}^{i},1\leq i\leq\widetilde{N}_{k}\}$.
moreover, the descendants of $z_{1},\dots,z_{P}$ at time $P$ are
independent variables.\end{lemma}

Let $\delta>0$. For all $k\in[P]$, let us take $H_{1}=G_{1},\dots,H_{k-1}=G_{k-1}$,
and for $j\geq k$, 
\[
H_{j}(x)=\begin{cases}
sup_{|y-z_{j}|<f^{\circ(j-k)}(\delta)}G_{k}(y) & \mbox{ if }|x-z_{j}|\leq f^{\circ(j-k)}(\delta)\\
G_{k}(y) & \mbox{ otherwise ,}
\end{cases}
\]
and let us note with ~$\widetilde{\,}$~~the corresponding system,
\[
\mbox{meaning }(\widetilde{X}_{k}^{i},\widetilde{A}_{k}^{i},\widetilde{B}_{k},\widetilde{\sigma}_{k},k\in\dots)=\Phi((z_{i})_{i\in[P]},\left(V_{\mathbf{u}},W_{\mathbf{u}}\right)_{\mathbf{u}\in(\N^{*})^{[k]},k\geq1},(H_{k})_{1\leq k\leq P}))\,.
\]
 Let $z'_{1},\dots,z'_{P}$ be such that $z'_{i}\in B_{\delta}(z_{i})$,
$\forall i$. We have 
\[
(X_{k}^{i},A_{k}^{i},B_{k},\sigma_{k},k\in\dots)=\Phi((z'_{i})_{i\in[P]},\left(V_{\mathbf{u}},W_{\mathbf{u}}\right)_{\mathbf{u}\in(\N^{*})^{[k]},k\geq1},(G_{k})_{1\leq k\leq P})\,.
\]
Using the Lemma above and Hypothesis \ref{hyp:1}, we have $N_{P}\leq\widetilde{N}_{P}$.
Let $\Phi'$ be such that 
\[
N_{P}=\Phi'((z_{i})_{i\in[P]},\left(V_{\mathbf{u}},W_{\mathbf{u}}\right)_{\mathbf{u}\in(\N^{*})^{[k]},k\geq1},(H_{k})_{1\leq k\leq P}))\,.
\]

\subsection{Examples}

\subsubsection{Gaussian example}

We draw sequences $(X_{k})_{k\in[P]}$, $(Y_{k})_{k\in[P]}$ such
that: $X_{1}\sim\mathcal{N}(0,1)$, $X_{k+1}=aX_{k}+bV_{k+1}$($a\in]0;1[$),
$Y_{k}=X_{k}+cW_{k}$ with i.i.d. variables $V_{k},W_{k}$ of law
$\mathcal{N}(0,1)$. We set 
\[
G_{k}(x)=\frac{1}{\sqrt{2\pi c^{2}}}\exp\left(-\frac{1}{2c^{2}}(x-Y_{k})^{2}\right),
\]
 $M_{1}(dx)=\frac{1}{\sqrt{2\pi}}e^{-x^{2}/2}dx$, $M(x,dy)=\frac{1}{\sqrt{2\pi b^{2}}}\exp\left(-\frac{(y-ax)^{2}}{2b^{2}}\right)dy$.
We want to bound. at time $P$, the particles descending from a fixed
trajectory. The descendants of different $z_{k}$ are independant
so we look, for all $k$, at which is the $z_{k}$ spawning the most
descendants at time $P$. Using the result above, we slice $E$ in
balls of size $\delta$. If $z_{k}'$ is in a ball of size $\delta$
containing $z_{k}$, the number of descendants of $z'_{k}$ at time
$P$ computed with potentials $G_{.}$ is bounded by the number of
descendants of $z_{k}$ at time $P$ computed with potentials $H_{.}$.
The potentials $G_{k}$ going to $0$ at $\pm\infty$, we do not have
to explore the whole of $\R$, as soon as $z_{k}$ is far enough from
$Y_{k}$ so that it has $0$ children under potential $H_{k}$, we
can stop the exploration.
\begin{rem}
With $\delta=0$ (or $\delta$ very small), if we look at the number
of descendants at time $P$ of an individual at time $k$ and we maximise
in the position of the individual, we will finite some finite quantity
(not exploding when $P-k\rightarrow+\infty$. For the maximisation
step, we have to take $\delta>0$ and then this maximum explodes (slowly).
So, there a balance to find between $\delta$ small (maximisation
step takes a lot of time) and $\delta$ big (explosion in the number
of particles). A rule of thumb, coming from the experience, is that
the population do not explode as long as the number of children per
individual is of order $2,3$. 
\end{rem}

\subsubsection{Directed polymers}

Let $(X_{k})_{k\geq1}$ be a symmetric simple random walk in $\Z$
with $X_{1}=0$. We are given i.i.d. variables $(\xi_{k,i})_{k\geq1,i\in\Z}$
with Bernoulli law of parameter $p>0$. We set (random) potentials
: $V_{k}(i)=\exp(-\beta\xi_{k,i})$ ($\beta>0$) and we are interested
in the following law (quenched, meaning the $\xi_{k,i}$ are fixed)
: 
\[
\pi_{1:k}(f)=\frac{\E_{\xi}(f(X_{1:k})\prod_{i=1}^{k}V_{i}(X_{i}))}{\E_{\xi}(\prod_{i=1}^{k}V_{i}(X_{i}))}\,.
\]
This kind of model is described in \cite{bezerra-tindel-viens-2008}.
If we take the expectation over all the variable: $\E(\mbox{max de la traj. sous }\pi_{1:n})$
behaves as $n^{\zeta}$ with $\zeta\neq1/2$.

Using our algorithm, we can simulate trajectories under the law $\pi_{1:P}$
(for fixed $\xi$, $P\in\N^{*}$). The research of the ancestors having
the biggest number of descendants at time $P$ makes that the computational
cost is $P^{2}.$ Here is the drawing of $\E(\max\dots)$ as a function
of $n$ in a $\log$-$\log$ scale (the blue line has gradient $2/3$,
the green line has gradient $1/2$): 
\begin{figure}[H]
\begin{centering}
\includegraphics[scale=0.7]{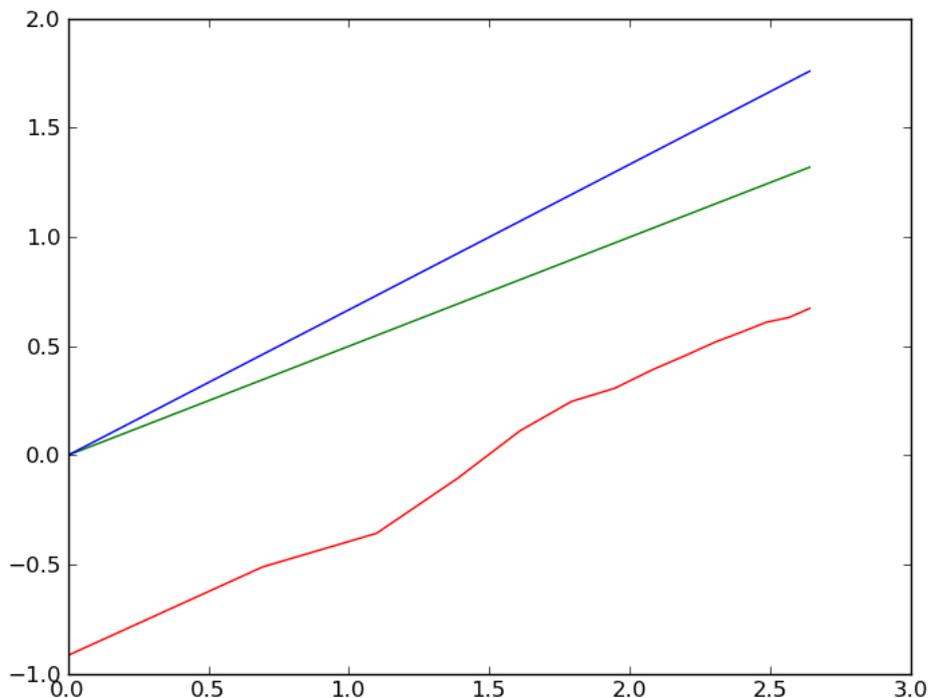}
\par\end{centering}

\caption{gradient estimation (least square)=0,63}

\end{figure}

\bibliographystyle{alpha}
\nocite{*}
\bibliography{simulation-parfaite}

\end{document}